\documentclass[14pt,dvipsnames]{extarticle}

\usepackage[left=2cm,top=2cm,right=2cm,bottom=20mm]{geometry}
\usepackage[T1]{fontenc}
\usepackage[utf8]{inputenc}
\usepackage[cmintegrals,cmbraces]{newtxmath}
\usepackage[lining]{ebgaramond}
\usepackage{ebgaramond-maths}
\usepackage{euscript}
\usepackage{mathtools}
\makeatletter
\DeclareSymbolFont{ntxletters}{OML}{ntxmi}{m}{it}
\SetSymbolFont{ntxletters}{bold}{OML}{ntxmi}{b}{it}
\re@DeclareMathSymbol{\partial}{\mathord}{ntxletters}{"40}
\re@DeclareMathSymbol{\wp}{\mathord}{ntxletters}{"7D}
\makeatother
\usepackage{graphicx}
\usepackage[figurename=Fig.]{caption}
\usepackage{ifpdf}
\ifpdf
  \usepackage[tracking=true,letterspace=50]{microtype}
  \SetTracking{encoding=T1,shape=sc}{50}
\else
  \usepackage[tracking=false]{microtype}
\fi

\usepackage{amsthm}
\usepackage{array,booktabs}
\usepackage{enumitem}
\usepackage{xcolor}
\usepackage{tocloft}
\usepackage{csquotes}

\usepackage[explicit]{titlesec}
\usepackage{pgfornament}

\titleformat{\section}[block]
  {\normalfont\Large\bfseries\filcenter} 
  {}
  {0pt}
  {\thesection.\hspace{0.45em}#1}

\titleformat{name=\section,numberless}[block]
  {\normalfont\Large\bfseries\filcenter}
  {}
  {0pt}
  {#1}

\titlespacing*{\section}
  {0pt}
  {3ex plus 0.3ex minus 0.2ex}
  {2.5ex}

\usepackage{hyperref}
\usepackage[nameinlink,capitalise,noabbrev]{cleveref}
\usepackage[
  backend=biber,
  style=alphabetic,
  sorting=nyt,
  giveninits=true,
  maxbibnames=99,
  doi=true,
  url=true,
  isbn=false,
  eprint=true
]{biblatex}
\AtBeginBibliography{\normalsize\raggedright\interlinepenalty=10000}
\definecolor{paperred}{RGB}{125,18,18}
\hypersetup{
  colorlinks=true,
  linkcolor={red!50!black},
  citecolor={blue!50!black},
  urlcolor={blue!80!black},
  pdfauthor={Mikhail Kornev},
  pdftitle={n-Valued Analogue of Abel's Theorem}
}

\setlist{nosep}
\allowdisplaybreaks
\theoremstyle{plain}
\newtheorem{theorem}{Theorem}[section]
\newtheorem{proposition}[theorem]{Proposition}
\newtheorem{lemma}[theorem]{Lemma}
\newtheorem{corollary}[theorem]{Corollary}
\theoremstyle{definition}
\newtheorem{definition}[theorem]{Definition}
\newtheorem{example}[theorem]{Example}

\theoremstyle{remark}

\makeatletter
\renewenvironment{proof}[1][\proofname]{
  \par\pushQED{\qed}\normalfont
  \topsep6\p@\@plus6\p@\relax
  \trivlist\item[\hskip\labelsep\bfseries\itshape #1\@addpunct{.}]
}{\popQED\endtrivlist\@endpefalse}
\makeatother

\let\phi\varphi

\DeclareMathOperator{\Sym}{Sym}
\DeclareMathOperator{\Tr}{Tr}
\DeclareMathOperator{\Res}{Res}
\DeclareMathOperator{\Disc}{Disc}
\DeclareMathOperator{\Spec}{Spec}
\DeclareMathOperator{\Frac}{Frac}
\DeclareMathOperator{\ord}{ord}
\DeclareMathOperator{\id}{id}
\DeclareMathOperator{\mult}{mult}

\newcommand{\Pp}{\mathbb P}
\newcommand{\kfield}{\Bbbk}
\newcommand{\Diff}{\kfield(t)\langle\partial_t\rangle}
\newcommand{\multiset}[1]{\mathopen{[}#1\mathclose{]}}
\newcommand{\dd}{\mathop{}\!\mathrm d}

\newcommand{\cT}{\mathcal T}
\newcommand{\cL}{\mathcal L}
\newcommand{\cR}{\mathcal R}
\newcommand{\muN}{\boldsymbol{\mu}}

\title{\bfseries \(n\)-Valued Analogue of Abel's Theorem}
\author{Mikhail Kornev}
\date{}

\renewbibmacro{in:}{}

\begin{document}

\maketitle

\begin{abstract}
Abel's classical theorem shows that a continuously differentiable one-dimensional associative law with a two-sided unit is locally equivalent to addition under a change of coordinate.  We prove an
\(n\)-valued analogue in the algebraic and formal settings.  Over a field of characteristic zero, every associative
algebraic \(n\)-valued monoid on \(\mathbb P^1\) with a strong two-sided unit is
commutative.  Over a reduced commutative \(\mathbb Q\)-algebra, any one-dimensional formal \(n\)-valued group is commutative.  We then classify \(n\)-valued
formal groups over arbitrary \(\mathbb Q\)-algebras that are integral domains: for any \(n\)-valued formal group there is a unique divisor \(m\) of \(n\) such that the law is strongly isomorphic to
the \(n/m\)-fold diagonal of the cyclic elementary \(m\)-valued law \(\mathbb
G_m\). This proves Kholodov's classification conjecture over integral domain \(\mathbb Q\)-algebras.
\end{abstract}



\tableofcontents

\thispagestyle{empty}

\newpage

\section{Introduction}\label{sec:introduction}

Starting from addition \(F(x,y)=x+y\) and multiplication \(F(x,y)=xy\),
Abel asked for all functions \(F\) for which \(F(z,F(x,y))\) is symmetric
in \(x,y,z\) \parencite[p.~11]{Abel1826}. Thus he sought laws with the
composition property of addition. In the locally nondegenerate continuously
differentiable setting (meaning that both first partial derivatives of \(F\)
are nonzero throughout a suitable local domain), his argument produces a locally
invertible function \(g\) such that
\[
 g(F(x,y))=g(x)+g(y).
\]

\begin{figure}[ht]
\centering
\includegraphics[width=0.96\textwidth]{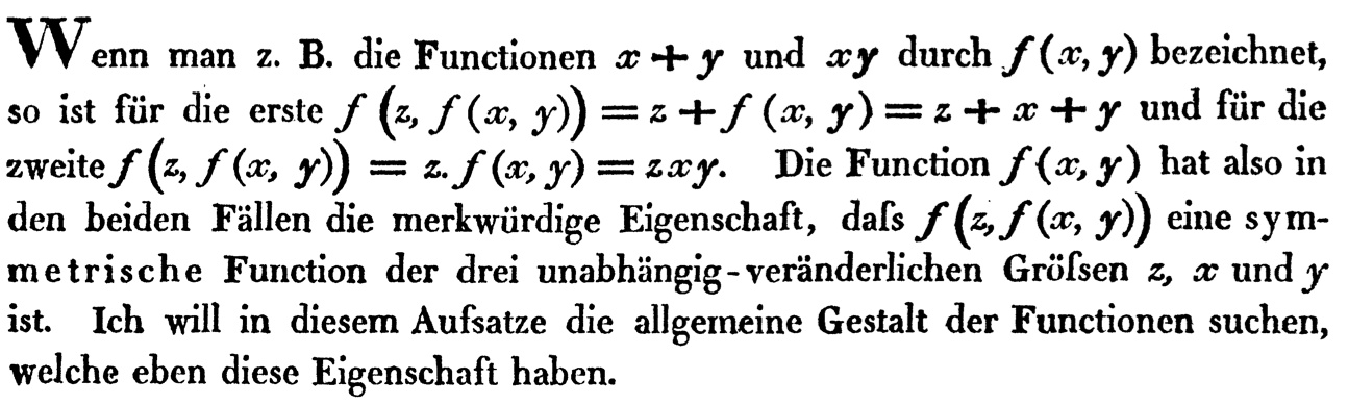}
\caption{The opening paragraph of Abel's 1826 paper \cite{Abel1826}.}
\label{fig:abel-opening}
\end{figure}

Writing \(f=g^{-1}\), the identity becomes
\(F(f(u),f(v))=f(u+v)\).
In the complex analytic setting, an algebraic law \(F\) therefore gives an
\emph{algebraic addition theorem} for \(f\):
\[
 \Phi(f(u),f(v),f(u+v))=0
\] for some nonzero \(\Phi\in\mathbb C[X,Y,Z]\).
Weierstrass proved that every nonconstant analytic function with such an
addition theorem is algebraic over one of
\(\mathbb C(u)\), \(\mathbb C(e^{cu})\) (where \(c\ne0\)), or
\(\mathbb C(\wp(u;\Lambda))\), where \(\Lambda\) is a lattice in
\(\mathbb C\) \parencite[Art.~1, p.~2]{WeierstrassSchwarz1885};
see also \parencite[p.~33]{Phragmen1885}.
This relates Abel's change of coordinate to the algebraic, exponential,
and elliptic addition theorems.

The characteristic-zero theorem for one-dimensional formal group laws is the
formal version of the same mechanism. If
\(a(t)=\partial_yF(t,0)\), differentiation of associativity at the unit gives
\[
 a(F(x,y))=\partial_yF(x,y)a(y).
\]
Since \(a(0)=1\), there is a unique formal series \(g\) with
\(g(0)=0\) and \(g'(t)=a(t)^{-1}\). The displayed identity gives
\(\partial_y g(F(x,y))=g'(y)\), so the unit identity \(F(x,0)=x\) yields
\[
 g(F(x,y))=g(x)+g(y).
\]
Thus \(g\) is the formal logarithm, and commutativity follows from
associativity and the two unit identities; see \cite{Hazewinkel1978}.

Recall that an \(n\)-valued multiplication on a space \(X\) assigns to \((x,y)\) an
unordered \(n\)-tuple, equivalently an \(n\)-multiset,
\[
 x*y=\multiset{z_1,\ldots,z_n}\in \Sym^n(X).
\]
Associativity means equality of the two flattened \(n^2\)-multisets obtained
from \((x*y)*z\) and \(x*(y*z)\). The identity and inverse axioms are stated
in the same language. This theory originated in algebraic topology in work of
Buchstaber and Novikov. Its algebraic apparatus was subsequently developed in
terms of coalgebras, generalized shifts, Frobenius \(n\)-homomorphisms, and
\(n\)-Hopf algebras
\parencite{BuchstaberNovikov1971,Buchstaber1975,BuchstaberRees1997,BuchstaberRees1998,BuchstaberRees2002,Buchstaber2006}.

The elliptic case gives a concrete link with multivalued laws. Identifying
\(u\) and \(-u\) on the elliptic curve \(\mathbb C/\Lambda\) gives a
two-valued law on \(\mathbb P^1\):
\[
 x*y=\multiset{\wp(u+v),\wp(u-v)},
 \qquad x=\wp(u),\quad y=\wp(v).
\]
The two outputs account for the choice of signs of the lifts, and the
strong unit is \(\infty\)
\parencite[\S~4.2]{BuchstaberGaiurRubtsov2026}.

In the classical commutative theory of two-valued formal groups, Buchstaber
introduced logarithms and exponentials. An \(n\)-valued formal law is called
\emph{elementary} if its monic defining polynomial \(P(Z; x, y)\) is homogeneous of total
degree \(n\) in \(Z,x,y\). When a series \(B=x+O(x^2)\) (called a strong change of coordinate) identifies a formal
multivalued law \(P(Z; x, y)\) with an elementary one, \(B\) is called its {\it logarithm} and
the inverse series \(B^{-1}\) its {\it exponential}. Buchstaber obtained a complete
classification of two-valued formal groups over \(\mathbb Q\)-algebras,
including the additional normal forms that occur over nonreduced \(\mathbb Q\)-algebras.
Over an integral domain \(\mathbb Q\)-algebra, he proved that every
two-valued formal group is strongly isomorphic to either the
cyclic elementary two-valued law or the two-fold diagonal of addition
\parencite[Theorem~6.4]{Buchstaber1975}; see also \cite{Buchstaber1976}.
Kholodov later extended the
method of generalized shift operators to commutative \(n\)-valued formal groups and proved that a
formal group is uniquely reconstructed from its infinitesimal generator. Over
a \(\mathbb Q\)-algebra without zero divisors, he classified three-valued
formal groups up to strong isomorphism. He also classified elementary
four- and five-valued formal groups. Finally, he proved that, for every
\(n\leq 11\), there are only finitely many elementary \(n\)-valued formal
groups. In a later paper, he established the divisor classification for
elementary formal groups of orders \(6\), \(7\), and \(8\)
\parencite[Theorem~7]{Kholodov1984}.
Kholodov conjectured that the elementary \(n\)-valued formal groups
are precisely the \((n/m)\)-fold diagonals of the cyclic elementary
\(m\)-valued groups, where \(m\mid n\)
\parencite[Conjecture~6.4]{Kholodov1982}. Our classification theorem below
proves this conjecture over an integral domain \(\mathbb Q\)-algebra and,
more generally, shows that every associative one-dimensional formal law
with a strong unit over such a base is strongly isomorphic to one of these laws.

Gaifullin showed that every involutive two-valued group---one satisfying
\(e\in x*y\) if and only if \(x=y\), where \(e\) is the unit---is commutative
\cite{Gaifullin_2-val_commutative}.

The axioms for an abstract \(n\)-valued group do not imply commutativity: an
ordinary noncommutative group gives a noncommutative \(n\)-valued group simply
by repeating its product \(n\) times.  There are also genuinely multivalued,
non-diagonal examples.  The following example, communicated by Alexander
Levine, is two-valued, non-reversible (its inverse relation is not
single-valued, in the terminology of \cite{Kornev2026OrderThree}), and
non-coset. Let \(X=\{0,1,2,3,4,5\}\).
In \cref{tab:levine}, the entry in row \(x\) and column \(y\) is \(x*y\).
The brackets denote a multiset, so \([i,i]\) is a double point.

\begin{table}[ht]
\centering
\small
\setlength{\tabcolsep}{5.2pt}
\renewcommand{\arraystretch}{1.16}
\begin{tabular}{c|cccccc}
\toprule
\(*\) & \(0\) & \(1\) & \(2\) & \(3\) & \(4\) & \(5\) \\
\midrule
\(0\) & \([0,0]\) & \([1,1]\) & \([2,2]\) & \([3,3]\) & \([4,4]\) & \([5,5]\) \\
\(1\) & \([1,1]\) & \([3,5]\) & \([0,3]\) & \([4,4]\) & \([0,5]\) & \([2,2]\) \\
\(2\) & \([2,2]\) & \([0,5]\) & \([3,5]\) & \([1,1]\) & \([0,3]\) & \([4,4]\) \\
\(3\) & \([3,3]\) & \([2,2]\) & \([4,4]\) & \([5,5]\) & \([1,1]\) & \([0,0]\) \\
\(4\) & \([4,4]\) & \([0,3]\) & \([0,5]\) & \([2,2]\) & \([3,5]\) & \([1,1]\) \\
\(5\) & \([5,5]\) & \([4,4]\) & \([1,1]\) & \([0,0]\) & \([2,2]\) & \([3,3]\) \\
\bottomrule
\end{tabular}
\caption{A noncommutative nondiagonal two-valued group on six elements.}
\label{tab:levine}
\end{table}

The element \(0\) is a strong unit, and a direct computer check of the \(6^3\)
triples verifies associativity. The multiplication admits more than one
inverse map. If
\(\operatorname{Inv}(x)\) denotes the set of all \(y\) for which both
\(x*y\) and \(y*x\) contain \(0\), then the table gives
\[
 \begin{gathered}
 \operatorname{Inv}(0)=\{0\},\qquad
 \operatorname{Inv}(1)=\{2,4\},\qquad
 \operatorname{Inv}(2)=\{1,4\},\\
 \operatorname{Inv}(3)=\{5\},\qquad
 \operatorname{Inv}(4)=\{1,2\},\qquad
 \operatorname{Inv}(5)=\{3\}.
 \end{gathered}
\]
The three two-element choices give eight inverse maps. Since
every coset \(n\)-valued group is reversible, it is also non-coset. It
is noncommutative because
\[
 1*2=[0,3],
 \qquad
 2*1=[0,5].
\]

The example shows that even a strong unit does not force commutativity in
the category of sets. It turns out that introducing geometric structures on \(n\)-valued groups changes the situation drastically.

In the setting of \(n\)-valued algebraic or formal groups, Buchstaber (for the \(n=2\) case) and Kholodov (for general \(n\)) noted that the normalized
trace is the correct object to work with. It defines a rational comultiplication
in the algebraic setting, a completed coalgebra in the formal setting, and two generalized
shift families \(\cL_a\) and \(\cR_b\). Associativity gives the mixed
commutation relation
\[
 \cL_a\cR_b=\cR_b\cL_a.
\]
On a one-dimensional algebraic or formal chart, the infinitesimal coefficients
of these families are ordinary differential operators. For our results, the main point is the Schur--Amitsur theorem: a non-scalar coefficient in the decomposition of \(\cR_b\) has a
commutative centralizer, so all coefficients in \(\cL_a\), and hence all left shifts,
commute. Evaluation at the strong unit converts this operator identity into
equality of power sums, and Newton's identities recover the multivalued
multiplication itself. In this way an associativity problem for algebraic
correspondences is reduced to the rigidity of centralizers of ordinary
differential operators.

The algebraic version of this \(n\)-valued Abel theorem is as follows.

\begin{theorem}[Commutativity of \(n\)-valued monoids on \(\Pp^1_{\kfield}\)]\label{thm:algebraic-commutativity}
Let \(\kfield\) be a field of characteristic zero. Every associative
algebraic \(n\)-valued monoid on \(\Pp^1_{\kfield}\) with a strong two-sided
unit is commutative.
\end{theorem}

No inverse axiom is required. The strong multiplicity at the unit is essential but
weak units (in the sense of \cref{def:weak-unit}) do not suffice. Associativity in the presence of a weak unit does not imply commutativity,
even for a weakly unital algebraic two-valued law on the affine line. Over any
field of characteristic different from two, set
\begin{equation}\label{eq:weak-unit-counterexample}
 x*y=[x+y,x-y].
\end{equation}
Both bracketings of a triple give
\[
 [x+y+z,x+y-z,x-y+z,x-y-z],
\]
so the law is associative. It is not commutative for generic \(x,y\). Since
\(x*0=[x,x]\) and \(0*x=[x,-x]\), the point \(0\) is a weak unit but not a
strong unit. Thus the full multiplicity condition at the identity is
essential. In Buchstaber's standard terminology the identity axiom is strong,
so \eqref{eq:weak-unit-counterexample} is not a two-valued monoid or group
\parencite{Buchstaber2006,BuchstaberKornev2025Wendt}. Further investigations of weak units are collected in \cref{sec:weak-units}.

The same local operator argument gives commutativity over every
characteristic-zero field in the setting of formal \(n\)-valued groups. The passage from fields to a
general coefficient ring is controlled by the nilradical. As we will see, the symmetry identity \(P(Z;x,y)=P(Z;y,x)\) can
be tested after reduction modulo every prime ideal and passage to the fraction
field. In particular, this equality descends when the intersection of the prime ideals is zero.
This yields the formal version of the same \(n\)-valued Abel theorem:

\begin{theorem}[Commutativity of \(n\)-valued formal groups over a reduced base]\label{thm:formal-commutativity}
Let \(R\) be a reduced commutative \(\mathbb Q\)-algebra. Let
\[
 P(Z;x,y)=Z^n-\theta_1(x,y)Z^{n-1}+\cdots+(-1)^n\theta_n(x,y)
 \in R[[x,y]][Z]
\]
be an associative one-dimensional formal \(n\)-valued multiplication with
strong two-sided unit
\[
 P(Z;x,0)=(Z-x)^n,
 \qquad
 P(Z;0,y)=(Z-y)^n.
\]
Then
\[
 P(Z;x,y)=P(Z;y,x),
\] i.e., the \(n\)-valued multiplication is commutative.
Under the same associativity and unit hypotheses over an arbitrary commutative
\(\mathbb Q\)-algebra, the coefficients of \(P(Z;x,y)-P(Z;y,x)\) belong to
the nilradical of \(R\).
\end{theorem}

Reducedness of the coefficient ring is what permits fieldwise commutativity to
descend. The classification does not require \(P\) to be generically squarefree: over a
domain we apply the branch argument to the squarefree support and retain each
root's multiplicity.  Combining \cref{thm:formal-commutativity} with
Buchstaber--Kholodov's finite-order generator theory then yields the full
divisor classification.

A basic example in the theories of algebraic and formal \(n\)-valued groups is the cyclic elementary \(n\)-valued group \(\mathbb G_n\), whose multiplication is encoded by the polynomial
\[
p_n(Z;x,y)
={}
\prod_{j=0}^{n-1}
\left(
Z-
\left(\sqrt[n]{x}+\zeta_n^j\sqrt[n]{y}\right)^n
\right)
\in \mathbb Z[Z,x,y],
\]
where \(\zeta_n\) is a primitive \(n\)-th root of unity. Thus the \(n\) values of \(x*_{\mathbb G_n}y\) are
\[
\left(\sqrt[n]{x}+\zeta_n^j\sqrt[n]{y}\right)^n,
\qquad j=0,\ldots,n-1.
\]Although this description uses \(n\)-th roots, the polynomial \(p_n(Z;x,y)\) has integral coefficients and therefore defines the law intrinsically. The algebraic and arithmetic properties of these groups and their relations with the theory of discriminants and the theory of projective duality were studied by Buchstaber and Kornev \cite{BuchstaberKornev2025Wendt, BuchstaberKornev2025CP1}. 

\begin{theorem}[Classification of \(n\)-valued formal groups over an integral domain \(\mathbb Q\)-algebra]\label{thm:cyclicity}
Let \(R\) be a commutative \(\mathbb Q\)-algebra which is an integral domain,
and let \(P(Z;x,y)\in R[[x,y]][Z]\) define an associative one-dimensional
formal \(n\)-valued multiplication with a strong two-sided unit. Then \(P\)
is strongly isomorphic over \(R\) to the \(d\)-fold diagonal of
\(\mathbb G_m\), for a unique divisor \(m\mid n\), where \(d=n/m\). Equivalently,
there is a strong coordinate \(B(t)=t+O(t^2)\in R[[t]]\) for which
\[
 \prod_{\nu=1}^{n}\left(W-B(z_\nu(x,y))\right)
 =p_m\bigl(W;B(x),B(y)\bigr)^d.
\]
More precisely, \(m\) is both the differential order (the rank) of the
generator and the number of distinct generic roots of \(P\).  The top
coefficient of the generator is
\[
 a_m(t)=\frac{m^m}{m!}t^{m-1}A(t),
 \qquad A(t)\in1+tR[[t]].
\]
There is a unique \(H(t)\in1+tR[[t]]\) satisfying
\(H(t)+mtH'(t)=A(t)^{-1/m}\).
Writing \(A(t)^{-1/m}=\sum_{k\ge0}g_kt^k\), the strong coordinate may be
chosen explicitly as
\[
 \begin{aligned}
 B(t)&=tH(t)^m
 =\frac1{m!}\left(\int_0^t a_m(u)^{-1/m}\,\dd u\right)^m\\
 &=t\left(\sum_{k\ge0}\frac{g_k}{1+mk}t^k\right)^m.
 \end{aligned}
\]
The integral is formal; its \(m\)-th power is independent of the root chosen.
If \(E=B^{-1}\), then after adjoining \(m\)-th roots and a primitive
\(m\)-th root of unity the distinct branches are
\[
 z_j(x,y)=E\left(
 \left(B(x)^{1/m}+\zeta_m^jB(y)^{1/m}\right)^m
 \right),
 \qquad j=0,\ldots,m-1,
\]
and each occurs with multiplicity \(d\).  In particular,
\(\Disc_Z(P)\ne0\) if and only if \(m=n\); the discriminant-zero cases are
exactly the proper diagonals \(d>1\).
\end{theorem}

Over the fraction field, write the generic polynomial using its distinct roots.
The top symbol of Kholodov's translation equation separates these roots and,
after a ramified change of coordinate, forces every distinct branch to have
the form \(u+\lambda v\), with \(\lambda^m=1\).  Because the original
coefficients depend on \(u\) and \(v\) only through \(u^m\) and \(v^m\), the
set of branches is invariant under multiplication by every \(m\)-th root of
unity.  This forces all \(m\) branches to occur, and to occur with one common
multiplicity \(d\).  Thus \(n=md\) and the transformed polynomial is
\(p_m^d\).

\begin{corollary}\label[corollary]{cor:divisor-count}
Under the hypotheses of \cref{thm:cyclicity}, the strong-isomorphism classes
of associative one-dimensional formal \(n\)-valued multiplications with a
strong two-sided unit are in bijection with the positive divisors of \(n\).
In particular, there are exactly \(\tau(n)\) such classes, where \(\tau(n)\) is the number of positive divisors of \(n\).
\end{corollary}

\begin{corollary}\label[corollary]{cor:elementary-laws}
Under the hypotheses of \cref{thm:cyclicity}, an elementary law satisfies
\(P(Z;x,y)=p_m(Z;x,y)^{n/m}\) for a unique divisor \(m\mid n\).
\end{corollary}

\begin{proof}
Homogeneity of the law gives \(a_i(t)=c_i t^{i-1}\) for its generator
coefficients. Thus \(A(t)=1\) in \cref{thm:cyclicity}, so \(H(t)=1\) and
\(B(t)=t\). The classification identity therefore holds in the original
coordinate, and uniqueness of \(m\) follows from the same theorem.
\end{proof}

The paper is organized as follows. \Cref{sec:definitions} fixes the common
notation and recalls the trace, shift, strong-isomorphism, and cyclic-law
constructions. \Cref{sec:aux-algebraic} contains the
auxiliary results and proof of \cref{thm:algebraic-commutativity} and
\cref{thm:formal-commutativity}.  The auxiliary results from Kholodov's theory
and Puiseux analysis needed for \cref{thm:cyclicity} are collected in
\cref{sec:aux-cyclicity}; its proof is given in
\cref{sec:proof-cyclicity}.  The final section discusses weak units, polynomial
pullbacks, and a projective-duality relation for the corresponding multivalued
laws.

\paragraph{Acknowledgements.} The author is grateful to Victor M. Buchstaber for useful discussions. The study has been funded within the framework of the HSE University Basic Research Program (HSE-BR-2025-060).

\section{Definitions and constructions}\label{sec:definitions}

We use the same polynomial convention in the algebraic and formal settings.
An unordered \(n\)-tuple is written as an \(n\)-multiset
\(
 \multiset{z_1,\ldots,z_n}
\), with repetitions recording multiplicities. An \(n\)-valued multiplication
on a space \(X\) is an operation
\[
 x*y=\multiset{z_1,\ldots,z_n}\in\Sym^n(X).
\]
Associativity means equality of the two flattened \(n^2\)-multisets obtained
from \((x*y)*z\) and \(x*(y*z)\). A point \(0\) is a \emph{strong two-sided
unit} if
\[
 0*x=x*0=\multiset{x,\ldots,x}.
\]
Commutativity means equality of multisets \(x*y=y*x\).

\subsection*{Algebraic laws on the projective line}
\begin{definition}
An \emph{algebraic \(n\)-valued monoid with a strong unit on
\(\Pp^1_{\kfield}\)} is a rational map
\[
 \mu:\Pp^1_{\kfield}\times\Pp^1_{\kfield}
 \dashrightarrow
 \Sym^n(\Pp^1_{\kfield})\cong\Pp^n_{\kfield}
\]
which is associative as a rational map and has a strong unit
\(e\in\Pp^1(\kfield)\). The unit condition means that \(\mu\) is defined at
the generic points of \(\{e\}\times\Pp^1\) and
\(\Pp^1\times\{e\}\), and that its restrictions there are the diagonal
degree-\(n\) divisors. Thus, after flattening symmetric powers,
\[
 (x*y)*z=x*(y*z)
 \quad\text{as rational maps to }\Sym^{n^2}(\Pp^1),
\]
\[
 x*e=e*x=\multiset{x,\ldots,x}.
\]
An algebraic \(n\)-valued group is an algebraic \(n\)-valued monoid equipped
with a rational inversion satisfying the usual multivalued inverse condition.
The commutativity theorem concerns the monoid before an inversion is chosen.
\end{definition}

After a projective change of coordinate, one may assume that the unit is the
finite point \(0\), and we make this choice from now on.  To encode the law on
the function field, write
\[
 \mathcal P(Z_0,Z_1;x,y)
 =\sum_{i=0}^n p_i(x,y)Z_0^{n-i}Z_1^i
\]
for a homogeneous binary form representing the generic effective divisor
\(x*y\).  The coefficient tuple
\[
 [p_0(x,y):\cdots:p_n(x,y)]\in\Pp^n(\kfield(x,y))
\]
is defined only up to a common nonzero factor.  On the affine output chart
\(Z_1=1\), the same divisor is represented by
\[
 P(Z;x,y)=p_0(x,y)Z^n+p_1(x,y)Z^{n-1}+\cdots+p_n(x,y).
\]
The coefficient \(p_0\) need not be constant.  For a representative regular
on an open set, a point where \(p_0=0\) and not all \(p_i\) vanish corresponds
to an output divisor containing the point at infinity.  Over the function
field \(F=\kfield(x,y)\) we may divide by \(p_0\) and put
\begin{equation}\label{hat_P}
 \widehat P(Z;x,y)=p_0(x,y)^{-1}P(Z;x,y).
\end{equation}
This division is made only in the function field.  It neither removes poles
nor extends the monoid multiplication across its indeterminacy locus.  The unit
condition gives
\begin{equation}\label{eq:unit-polynomial}
 \widehat P(Z;0,t)=\widehat P(Z;t,0)=(Z-t)^n
\end{equation}
in \(\kfield(t)[Z]\).

\subsection*{Formal laws, generalized translations, and strong coordinates}\label{generalized_translations}

Let \(R\) be a commutative \(\mathbb Q\)-algebra. A one-dimensional formal
\(n\)-valued multiplication is encoded by a monic polynomial
\begin{equation}\label{eq:formal-P}
 P(Z;x,y)
 =Z^n-\theta_1(x,y)Z^{n-1}+\cdots+(-1)^n\theta_n(x,y)
 \in R[[x,y]][Z].
\end{equation}
Its roots \(z_1(x,y),\ldots,z_n(x,y)\) are auxiliary branch symbols in a
splitting extension; every symmetric polynomial expression in them is
determined by the coefficients of \(P\). The strong two-sided unit identities are
\begin{equation}\label{eq:formal-unit}
 P(Z;x,0)=(Z-x)^n,
 \qquad
 P(Z;0,y)=(Z-y)^n.
\end{equation}
In particular, \(\theta_i(0,0)=0\) for every \(i\). We interpret a symmetric
formal series in the branch symbols by expressing it as a formal series
in their elementary symmetric functions and then substituting
\(\theta_1(x,y),\ldots,\theta_n(x,y)\). Each coefficient in \(x,y\) is
determined by finitely many terms, since all the \(\theta_i\) have zero
constant term.

Associativity is equality of the two iterated degree-\(n^2\) branch
polynomials.  A branch-free form of this identity is
\begin{equation}\label{eq:formal-associativity-resultant}
 \Res_U\bigl(P(U;x,y),P(W;U,z)\bigr)
 =
 \Res_V\bigl(P(V;y,z),P(W;x,V)\bigr)
 \quad\text{in }R[[x,y,z]][W].
\end{equation}

The {\it normalized formal translation} is
\begin{equation}\label{eq:formal-translation}
 (\cT_yf)(x)
 =\frac1n\sum_{z_{\nu}\in x\ast y} f\bigl(z_\nu(x,y)\bigr),
 \qquad f\in R[[x]],
\end{equation} where \(x\ast y = [z_{\nu}\mid P(z_{\nu}; x, y) = 0]\).
The expression is symmetric in the branches \(z_{\nu}(x, y)\) and is therefore well defined.
Thus \((\Delta f)(x,y)=(\cT_yf)(x)\) defines a comultiplication with target
\(R[[x]]\widehat\otimes_R R[[y]]=R[[x,y]]\); associativity gives
coassociativity, and \(f\mapsto f(0)\) is its counit.
Fixing one input gives the formal {\it left} and {\it right generalized shifts}
\begin{equation}\label{eq:formal-shifts}
 (\cL_af)(t)=\frac1n\sum_{z\in a*t}f(z),
 \qquad
 (\cR_bf)(t)=\frac1n\sum_{z\in t*b}f(z).
\end{equation}
Their operator coefficients will be studied in \Cref{sec:aux-algebraic}.

The cyclic elementary formal law \(\mathbb G_n\) is represented by the polynomial
\begin{equation}\label{eq:pn}
 p_n(W;X,Y)
 =\prod_{j=0}^{n-1}
 \left(W-\bigl(u+\zeta_n^jv\bigr)^n\right),
 \qquad u^n=X,\quad v^n=Y.
\end{equation}
The product belongs to \(\mathbb Z[W,X,Y]\) and is independent of the
auxiliary choices. These polynomials are expressed via generalized Wendt matrices in the paper \cite{BuchstaberKornev2025Wendt}.

If a \(k\)-valued law is represented by \(Q(W;X,Y)\), its
\emph{\(d\)-fold diagonal} is the \(dk\)-valued law
represented by \(Q(W;X,Y)^d\): every value, including its multiplicity in
\(Q\), is repeated \(d\) times.  In particular, when \(k\mid n\), the
\(n\)-valued diagonal of \(\mathbb G_k\) is represented by
\begin{equation}\label{eq:diagonal-cyclic}
 p_k(W;X,Y)^{n/k}.
\end{equation}
If the original branches are \(z_1,\ldots,z_k\), the normalized translation
of the repeated law is
\begin{equation}\label{eq:translation-repetition}
 \frac{1}{dk}\sum_{j=1}^{k}\sum_{\ell=1}^{d}f(z_j)
 =\frac{1}{k}\sum_{j=1}^{k}f(z_j).
\end{equation}
Thus repeating every branch the same number of times leaves the normalized
translation unchanged.

\begin{definition}\label[definition]{def:strong}
Let \(P(W; X, Y)\) and \(Q(W;X,Y)\) be monic formal \(n\)-valued laws, and let
\[
 B(t)=t+O(t^2)\in R[[t]].
\]
Thus \(B(0)=0\), \(B'(0)=1\), and \(B\) has a compositional inverse.
The series \(B\) is a \emph{strong isomorphism} from \(P\) to \(Q\) if
\begin{equation}\label{eq:strong-isomorphism}
 \prod_{\nu=1}^n
 \left(W-B\bigl(z_\nu(x,y)\bigr)\right)
 =Q\bigl(W;B(x),B(y)\bigr),
\end{equation} where, as usual, \(z_\nu(x, y)\) denotes formal roots of the polynomial \(P(z; x, y)\) with respect to \(z\).
The left-hand side of \eqref{eq:strong-isomorphism} is a polynomial in \(W\) with coefficients in
\(R[[x,y]]\), interpreted by the preceding convention for symmetric formal
series; every branch \(z_\nu(x, y)\) is counted with its multiplicity. Since \(B(t)\) has a compositional inverse, distinct branches remain distinct after applying \(B(t)\): each branch is repeated the same number of times in the transformed polynomial as in the original polynomial.  \Cref{def:strong} is equivalent to
Buchstaber's;
see \parencite[Definition~1.3]{Buchstaber1975} and
\parencite[Lemma~1.3]{Kholodov1982}.
\end{definition}

\begin{lemma}[Commuting-shift criterion]\label[lemma]{lem:commuting-shifts}
\begin{enumerate}[label=\textup{(\roman*)},leftmargin=*]
\item Let \(X\) be a set equipped with an associative \(n\)-valued multiplication
\[
 \mu:X\times X\longrightarrow\Sym^n(X)
\]
and a strong two-sided unit \(e\). Fix a field \(\kfield\) of characteristic
zero. For each \(a\in X\), define the normalized left shift on functions
\(f:X\to\kfield\) by
\[
 (\cL_a f)(x)=\frac1n\sum_{z\in\mu(a,x)}f(z),
\]
where the sum counts multiplicities. The multiplication \(\mu\) is
commutative if and only if its normalized left shifts commute pairwise:
\[
 \cL_a\cL_b=\cL_b\cL_a
 \qquad\text{for all }a,b\in X.
\]

\item Suppose that either \(\mu\) is an associative algebraic \(n\)-valued
multiplication on \(\Pp^1_\kfield\) over a field \(\kfield\) of characteristic
zero, or \(P(Z;x,y)\in R[[x,y]][Z]\) defines an associative one-dimensional
formal \(n\)-valued multiplication over a commutative \(\mathbb Q\)-algebra
\(R\). In either case, assume that \(0\) is a strong two-sided unit.
The multiplication is commutative if and only if its normalized left shifts
commute on rational functions \(\kfield(t)\), respectively formal power series
\(R[[t]]\). The commutation identities are understood in \(\kfield(a,b,t)\),
respectively \(R[[a,b,t]]\), for independent parameters \(a,b\).
\end{enumerate}
\end{lemma}

\begin{proof}
\textup{(i)} Suppose first that the normalized left shifts commute. The strong-unit
identity gives
\[
 (\cL_a\cL_bf)(e)
 =(\cL_bf)(a)
 =\frac1n\sum_{z\in\mu(b,a)}f(z).
\]
Interchanging \(a\) and \(b\) and using commutation yields
\[
 \sum_{z\in\mu(a,b)}f(z)=\sum_{z\in\mu(b,a)}f(z)
\]
for every \(f:X\to\kfield\). Taking \(f\) to be the indicator function
of each singleton shows that every element occurs with the same multiplicity
in the two multisets. Hence \(\mu(a,b)=\mu(b,a)\).

Conversely, suppose that \(\mu\) is commutative. Then
\[
 \cL_a\cL_b
 =\frac1n\sum_{c\in\mu(b,a)}\cL_c
 =\frac1n\sum_{c\in\mu(a,b)}\cL_c
 =\cL_b\cL_a.
\]
The first and last equalities follow from associativity, and the middle
equality follows from commutativity.

\textup{(ii)} In the one-dimensional algebraic and formal settings, the same
evaluation argument applies with \(f(t)=t^j\), \(1\le j\le n\), whenever
the left shifts commute on rational or formal functions, respectively.
Equality of these power sums and Newton's identities give equality of the
monic branch polynomials, hence commutativity of the multiplication.
The converse follows from the same associativity identity used in
\textup{(i)}.
\end{proof}

\section{Commutativity}\label{sec:aux-algebraic}

Throughout this section, \(\kfield\) is a field of characteristic zero. The
trace construction becomes effective through its expansion at the strong unit.

For \(f=g/h\in\kfield(t)\), the generic sum
\(\sum_{z\in x*y}f(z)\) can be understood as a trace of some linear operator.  Set
\[
 \mathcal A=F[Z]/(\widehat P(Z;x,y)),\quad F := \kfield(x, y),
\] see \eqref{hat_P} for notation.
Then \(\mathcal A\) is an \(n\)-dimensional \(F\)-algebra.  If \(f=g/h\),
the class of \(h(Z)\) is invertible in \(\mathcal A\).  Indeed,
\[
 \Res_Z\bigl(\widehat P(Z;0,y),h(Z)\bigr)
 =\Res_Z\bigl((Z-y)^n,h(Z)\bigr)=h(y)^n\ne0.
\]
Here \(h(y)\ne0\) because \(h\) is a nonzero polynomial and \(y\) is
transcendental over \(\kfield\).  Thus \(\Res_Z(\widehat P,h)\ne0\) in
\(F\).  B\'ezout's identity shows that
\(h(Z)\) is a unit modulo \(\widehat P\), so \(f(Z)\in\mathcal A\) is well
defined. 

For each \(u\in\mathcal A\), let
\[
 \mult_u:\mathcal A\longrightarrow\mathcal A
\]
denote multiplication by \(u\).

\begin{proposition}[Trace transformation and generalized shifts]\label[proposition]{prop:coalgebra}
One \(n\)-th of the trace,
\begin{equation}\label{eq:coproduct}
 (\Delta f)(x,y)
 =\frac1n\Tr_{\mathcal A/F}\bigl(\mult_{f(Z)}\bigr)
 =\frac1n\sum_{z\in x*y}f(z)
\end{equation}
defines a linear map \(\Delta:\kfield(t)\to\kfield(x,y)\).  Whenever the
specializations are defined, fixing the first or the second input gives the
generalized shifts
\begin{equation}\label{eq:shifts}
 (\cL_af)(t)=(\Delta f)(a,t),
 \qquad
 (\cR_bf)(t)=(\Delta f)(t,b)
\end{equation}
as rational operator families.  They satisfy \(\cL_a\cR_b=\cR_b\cL_a\) on their
common domain.  The unit gives
\begin{equation}\label{eq:shift-unit}
 \cL_0=\cR_0=\id,
 \qquad (\cL_af)(0)=(\cR_af)(0)=f(a).
\end{equation}
The last two evaluations hold whenever \(f\) is regular at \(a\).
\end{proposition}

\begin{proof}
In the basis \(1,Z,\ldots,Z^{n-1}\), the operator \(\mult_Z\) is represented by
the companion matrix of \(\widehat P\).  Over a splitting field its
eigenvalues are the elements of \(x*y\), counted with multiplicity.  Hence
\(\Tr(\mult_{f(Z)})=\sum_{z\in x*y}f(z)\) and
\eqref{eq:coproduct} defines the stated rational function.

Associativity gives the coassociativity identity
\begin{equation}\label{eq:rational-coassociativity}
 \frac1{n^2}\sum_{u\in x*y}\sum_{v\in u*z}f(v)
 =\frac1{n^2}\sum_{u\in y*z}\sum_{v\in x*u}f(v)
 \quad\text{in }\kfield(x,y,z),
\end{equation}
because its two sides are one \(n^2\)-th of the trace over \((x*y)*z\) and
\(x*(y*z)\), respectively. We call \(\Delta\) the rational trace
comultiplication of the law: coassociativity is understood as the identity
\eqref{eq:rational-coassociativity} in \(\kfield(x,y,z)\). Its target is
\(\kfield(x,y)=\Frac(\kfield(x)\otimes_{\kfield}\kfield(y))\).
Setting either input equal to \(0\) gives \eqref{eq:shift-unit}; evaluation
\(f\mapsto f(0)\) is defined only for rational functions regular at \(0\).
The same associativity identity with inputs
\((a,t,b)\) identifies the sums defining \(\cL_a\cR_bf\) and \(\cR_b\cL_af\).  Thus
\(\cL_a\cR_b=\cR_b\cL_a\).  
\end{proof}

The trace transformation becomes useful through its expansion at the unit.
The following local calculation shows that every coefficient of such an
expansion is an ordinary differential operator with rational coefficients.

\begin{proposition}[Trace coefficients are differential operators]\label[proposition]{prop:trace-differential}
Let \(K=\kfield(t)\) and let \(Q(Z;s)\in K[[s]][Z]\) be monic of degree
\(n\) in \(Z\), with \(Q(Z;0)=(Z-t)^n\).  Put
\(\mathcal A=K[[s]][Z]/(Q)\) and define
\begin{equation}\label{eq:local-trace-family}
 (\cT_sf)(t)
 =\frac1n\Tr_{\mathcal A/K[[s]]}\bigl(\mult_{f(Z)}\bigr),
 \qquad f\in K.
\end{equation}
If \(\cT_s=\sum_{r\geq0}s^rA_r\), then
\begin{equation}\label{eq:local-differential-coefficients}
 A_r\in \Diff,
 \qquad [\partial_t,t]=1,
\end{equation}
and its differential order is at most \(n(r+1)-1\).
\end{proposition}

\begin{proof}
The trace family is well defined for every \(f=g/h\in K\). Indeed,
\(\mathcal A/s\mathcal A\simeq K[\delta]/(\delta^n)\), where
\(\delta=Z-t\). In this quotient, \(h(Z)=h(t)+\delta u\) for some \(u\),
so it is invertible: \(h(t)\ne0\) and \(\delta\) is nilpotent.
Since \(\mathcal A\) is finite free over \(K[[s]]\), the matrix of
multiplication by \(h(Z)\) has determinant with nonzero constant term.
It is therefore invertible over \(K[[s]]\), and \(h(Z)\) is a unit in
\(\mathcal A\).

Fix \(r\geq0\), put \(R_r=K[s]/(s^{r+1})\) and
\(\mathcal A_r=R_r[Z]/(Q)\), and write \(\delta=Z-t\).  For some
\(H\in K[[s]][\delta]\), the relation in \(\mathcal A_r\) is
\(Q(t+\delta;s)=\delta^n+sH=0\).  Consequently,
\begin{equation}\label{eq:nilpotent-delta}
 \delta^{n(r+1)}=(-sH)^{r+1}=0
 \qquad\text{in }\mathcal A_r,
\end{equation}
because \(s^{r+1}=0\) in \(R_r\).

Taylor's formula is therefore finite in \(\mathcal A_r\):
\begin{equation}\label{eq:finite-taylor}
 f(Z)=f(t+\delta)
 =\sum_{j=0}^{n(r+1)-1}
 \frac{f^{(j)}(t)}{j!}\,\delta^j.
\end{equation}
For a rational function this identity follows by applying the polynomial
Taylor formula to its numerator and expanding its denominator. The latter is
invertible modulo the nilpotent element \(\delta\).
Taking one \(n\)-th of the trace of multiplication by both sides yields
\begin{equation}\label{eq:trace-taylor}
 \cT_sf(t)
 \equiv
 \sum_{j=0}^{n(r+1)-1}c_j(s,t)f^{(j)}(t)
 \pmod{s^{r+1}},
 \qquad
 c_j(s,t)
 =\frac{1}{n\,j!}
 \Tr_{\mathcal A_r/R_r}\bigl(\mult_{\delta^j}\bigr)
 \in R_r.
\end{equation}
Although \(\delta^j\) depends on \(Z\), its trace does not: the trace of an
\(R_r\)-linear endomorphism lies in the base ring \(R_r\). Taking the
coefficient of \(s^r\) in \eqref{eq:trace-taylor} expresses \(A_r\) as a
finite sum of operators \(a_j(t)\partial_t^j\), with
\(a_j(t)\in\kfield(t)\), and proves the assertion.
\end{proof}

The operators obtained from the two shift families will commute with a common
nonscalar operator. The final ingredient for the proof of \cref{thm:algebraic-commutativity} says that its centralizer is already
commutative.

\begin{theorem}[Schur--Amitsur centralizer theorem]\label{thm:schur-amitsur}
Let \(\kfield\) be a field of characteristic zero, let
\(K\) be either \(\kfield(t)\) or \(\kfield((t))\), and put
\(\mathcal D_K=K\langle\partial_t\rangle\). If
\(D\in\mathcal D_K\setminus\kfield\), then its centralizer
\[
 C_{\mathcal D_K}(D)
 =\{A\in\mathcal D_K:[A,D]=0\}
\]
is commutative.
\end{theorem}

If \(D\) has positive differential order, the assertion is Amitsur's theorem
for the differential field \(K\) \parencite[Theorem~1]{Amitsur1958}. If
\(D\) has order zero, it is multiplication by a nonconstant element
\(h\in K\). Comparing the leading term of \([A,h]\) shows that no
positive-order operator \(A\) commutes with \(h\), because \(h'\ne0\) in
characteristic zero. Its centralizer is therefore \(K\), which is
commutative. The positive-order phenomenon goes back to Schur
\parencite{Schur1905}. For centralizers in the polynomial-coefficient Weyl
algebra see \cite{MakarLimanov2021}.


\begin{proof}[Proof of \cref{thm:algebraic-commutativity}]

Choose a projective coordinate in which the unit is \(0\), and use the affine
chart containing it.  At the generic point of this chart, divide the polynomial
\(P(Z;x,y)\) by its leading coefficient and define the generalized shifts
\(\cL_a,\cR_b\) by \eqref{eq:shifts}. The identity
\(0*t=[t,\ldots,t]\) controls the expansion of \(\cL_a\) at \(a=0\), while
\(t*0=[t,\ldots,t]\) controls the expansion of \(\cR_b\) at \(b=0\). The
coefficients of the corresponding monic representatives are regular in these
parameters at zero: the rational map is defined at the generic points of the
two unit divisors and lands in the chosen affine output chart there. We
therefore have presentations
\begin{equation}\label{eq:operator-expansions}
 \cL_a=\sum_{r\geq0}a^rD_r(t,\partial_t),
 \qquad
 \cR_b=\sum_{s\geq0}b^sE_s(t,\partial_t)
\end{equation}
as formal operator families on \(\kfield(t)\), where \(D_r\) and \(E_s\) are
certain differential polynomials with rational coefficients.

Apply \cref{prop:trace-differential} to
\(Q(Z;s)=\widehat P(Z;s,t)\) and to
\(Q(Z;s)=\widehat P(Z;t,s)\).  The unit identity
\eqref{eq:unit-polynomial} supplies the required specialization at \(s=0\),
so
\begin{equation}\label{eq:coefficients-diff}
 D_r,E_s\in\Diff
 \qquad\text{for all }r,s\geq0.
\end{equation}

By \cref{prop:coalgebra}, associativity gives \(\cL_a\cR_b=\cR_b\cL_a\).
Substitute \eqref{eq:operator-expansions} and compare the coefficient of
\(a^rb^s\).  This gives
\begin{equation}\label{eq:mixed-coefficients}
 [D_r,E_s]=0
 \qquad\text{for all }r,s\geq0.
\end{equation}

Since \(\cR_b1=1\), comparison of coefficients gives \(E_s1=0\) for
every \(s\ge1\). Put \(f_0(t)=t\) and
\(H(t,b)=(\cR_bf_0)(t)-t\). The left-unit identity gives \(H(0,b)=b\),
so \(H\) is a nonzero rational function. Its formal expansion at \(b=0\)
is therefore nonzero, since a rational function with zero expansion is zero.
As \(E_0=\id\), the constant term cancels, leaving
\begin{equation}\label{eq:nonzero-shift-coefficient}
 H(t,b)=\sum_{s\ge1}b^s(E_sf_0)(t).
\end{equation}
Thus at least one coefficient \(E_{s_0}f_0\), with \(s_0\ge1\), is
nonzero. Consequently, \(E_{s_0}\ne0\), and \(E_{s_0}1=0\) shows that
it is not scalar.

Every \(D_r\) belongs, by \eqref{eq:mixed-coefficients}, to the centralizer of
\(E_{s_0}\) in \(\Diff\). By \cref{thm:schur-amitsur}, that centralizer is
commutative.  Therefore
\begin{equation}\label{eq:left-coefficients-commute}
 [D_r,D_q]=0
 \qquad\text{for all }r,q\geq0.
\end{equation}
It follows first as a formal identity at \((a,b)=(0,0)\). Expansion in the
local parameters is injective on rational functions regular there, so this is
an identity of rational operator families: all left generalized shifts
commute,
\begin{equation}\label{eq:left-shifts-commute}
 \cL_a\cL_b=\cL_b\cL_a.
\end{equation}

By \cref{lem:commuting-shifts}\textup{(ii)}, the multiplication is
commutative on the dense
open set where the correspondences and their affine branch polynomials are
defined.  Hence the two rational maps \(\mu(x,y)\) and \(\mu(y,x)\) to the
separated variety \(\Sym^n(\Pp^1)\) agree as rational maps.  This proves
\cref{thm:algebraic-commutativity}.

\end{proof}


The same argument applies to formal \(n\)-valued groups.

\begin{proposition}[Formal shift coefficients]\label[proposition]{prop:formal-shift-coefficients}
Let \(\kfield\) be a field of characteristic zero and let
\(P(Z;x,y)\in\kfield[[x,y]][Z]\) be an associative formal
\(n\)-valued multiplication satisfying \eqref{eq:formal-unit}. Then
\[
 \cL_a=\sum_{r\ge0}a^rD_r(t,\partial_t),
 \qquad
 \cR_b=\sum_{s\ge0}b^sE_s(t,\partial_t),
\]
where
\[
 D_r,E_s\in\kfield((t))\langle\partial_t\rangle.
\]
Moreover,
\[
 [D_r,E_s]=0
 \qquad\text{for all }r,s\ge0,
\]
and, for every \(f\in\kfield[[t]]\),
\[
 \cL_0=\cR_0=\id,
 \qquad
 (\cL_af)(0)=(\cR_af)(0)=f(a).
\]
\end{proposition}

\begin{proof}
For \(r\ge0\), reduce \(P(Z;a,t)\) modulo \(a^{r+1}\) and work in
\[
 \mathcal A_r=
 \bigl(\kfield[[t]][a]/(a^{r+1})\bigr)[Z]/(P(Z;a,t)).
\]
Writing \(\delta=Z-t\), the unit identity gives
\(\delta^n\in a\mathcal A_r\), hence \(\delta^{n(r+1)}=0\). The finite
Taylor expansion used in \cref{prop:trace-differential}, followed by the
normalized trace and extraction of the coefficient of \(a^r\), gives
\(D_r\in\kfield[[t]]\langle\partial_t\rangle\), of order at most
\(n(r+1)-1\). Applying the same argument to \(P(Z;t,b)\) gives the
assertion for \(E_s\). Associativity identifies the branch sums defining
\(\cL_a\cR_bf\) and \(\cR_b\cL_af\) in \(\kfield[[a,b,t]]\).
Comparing coefficients of \(a^rb^s\) gives \([D_r,E_s]=0\). The remaining
identities follow from the strong two-sided unit.
\end{proof}

\begin{corollary}[Formal commutativity over a field]
\label[corollary]{cor:formal-field-commutativity}
Let \(\kfield\) be a field of characteristic zero. Every associative
one-dimensional formal \(n\)-valued multiplication over \(\kfield\) with a
strong two-sided unit is commutative.
\end{corollary}

\begin{proof}
For \(f_0(t)=t\), the series \((\cR_bf_0)(t)\) belongs to
\(\kfield[[t,b]]\), so coefficient extraction commutes with evaluation at
\(t=0\). The unit identity therefore gives \((E_1f_0)(0)=1\).
Moreover, \(E_1(1)=0\), since \(\cR_b1=1\), so \(E_1\) is nonscalar.
By \cref{prop:formal-shift-coefficients}, every \(D_r\) belongs to its
centralizer in \(\kfield((t))\langle\partial_t\rangle\), which is
commutative by \cref{thm:schur-amitsur}. Hence all \(D_r\) commute, and
therefore so do the left shifts on \(\kfield[[t]]\).
\Cref{lem:commuting-shifts}\textup{(ii)} proves the claim.
\end{proof}


\begin{proof}[Proof of \cref{thm:formal-commutativity}]

Write
\[
 \delta_i(x,y)
 =\theta_i(x,y)-\theta_i(y,x)
 =\sum_{r,s\ge0}c_{i,r,s}x^ry^s,
 \qquad c_{i,r,s}\in R.
\]
We prove that every coefficient \(c_{i,r,s}\) belongs to the nilradical of
\(R\).

Fix a prime ideal \(\mathfrak p\in\Spec R\), and put
\[
 \kappa(\mathfrak p)=\Frac(R/\mathfrak p).
\]
Let
\[
 \lambda_{\mathfrak p}\colon R\longrightarrow\kappa(\mathfrak p)
\]
be the canonical homomorphism. It induces the coefficientwise homomorphism
\[
 \lambda_{\mathfrak p,*}\colon R[[x,y]][Z]
 \longrightarrow\kappa(\mathfrak p)[[x,y]][Z],
 \qquad
 \sum a_{r,s,j}x^ry^sZ^j
 \longmapsto
 \sum \lambda_{\mathfrak p}(a_{r,s,j})x^ry^sZ^j.
\]
Set
\[
 P_{\mathfrak p}(Z;x,y)=\lambda_{\mathfrak p,*}\bigl(P(Z;x,y)\bigr).
\]
The associativity identity \eqref{eq:formal-associativity-resultant} and the
unit identities are coefficientwise polynomial identities. Applying
\(\lambda_{\mathfrak p,*}\) to both sides preserves these identities.
Consequently \(P_{\mathfrak p}\) is an associative formal \(n\)-valued
multiplication over \(\kappa(\mathfrak p)\) with a strong two-sided unit.
Moreover,
\(\kappa(\mathfrak p)\) has characteristic zero because \(R\) is a
\(\mathbb Q\)-algebra. Hence
\cref{cor:formal-field-commutativity} gives
\[
 P_{\mathfrak p}(Z;x,y)=P_{\mathfrak p}(Z;y,x).
\]

Comparing the coefficient of \(Z^{n-i}x^ry^s\) in this equality gives
\[
 \lambda_{\mathfrak p}(c_{i,r,s})=0.
\]
The kernel of \(\lambda_{\mathfrak p}\) is exactly \(\mathfrak p\), so
\(c_{i,r,s}\in\mathfrak p\). Since \(\mathfrak p\) was arbitrary,
\[
 c_{i,r,s}\in
 \bigcap_{\mathfrak p\in\Spec R}\mathfrak p
 =\sqrt{(0)}.
\]
Thus every coefficient of \(P(Z;x,y)-P(Z;y,x)\) belongs to the nilradical of
\(R\). If \(R\) is reduced, then \(\sqrt{(0)}=0\), so all
\(c_{i,r,s}\) vanish and
\[
 P(Z;x,y)=P(Z;y,x).
\]
This proves \cref{thm:formal-commutativity}.

\end{proof}

\section{Auxiliary results for the classification}\label{sec:aux-cyclicity}

Let \(R\) be a commutative \(\mathbb Q\)-algebra which is an integral domain
and put \(K=\Frac(R)\). By \cref{thm:formal-commutativity}, every associative
formal \(n\)-valued multiplication over \(R\) with a strong two-sided unit is
commutative. Kholodov's generator theory is therefore available without an
additional commutativity hypothesis.

The generator is the coefficient of the translation parameter in
\eqref{eq:formal-translation}:
\begin{equation}\label{eq:generator-definition}
 (Df)(t)=\left.\frac{\partial}{\partial y}(\cT_yf)(t)\right|_{y=0},
 \qquad f\in R[[t]].
\end{equation}
We use the following facts from Kholodov's generator theory. First, \(D\)
is defined over \(R\) and has finite differential order at most \(n\):
\begin{equation}\label{eq:generator}
 D=\sum_{i=1}^{m}a_i(t)\frac{\dd^i}{\dd t^i},
 \qquad 1\le m\le n,
 \qquad a_m\ne0;
\end{equation}
see \parencite[Theorem~2.2]{Kholodov1982}.  We call \(m\) the \emph{rank} of
the law.  Second, invariance of \(D\) with respect to \(\cT_y\) gives, for
\(u(x,y)=\cT_yf(x)\),
\begin{equation}\label{eq:translation-equation}
 D_xu(x,y)=D_yu(x,y),
\end{equation}
where
\[
 D_x=\sum_{i=1}^{m}a_i(x)\partial_x^i,
 \qquad
 D_y=\sum_{i=1}^{m}a_i(y)\partial_y^i.
\]
The Cauchy datum is \(u(x,0)=f(x)\); compare
\parencite[Theorem~3.9]{Kholodov1982}. Third, the opening display in
Kholodov's proof of Theorem~3.9 \parencite[p.~276]{Kholodov1982} gives
\[
 a_i(t)=\frac{a_i^{(i-1)}(0)}{(i-1)!}t^{i-1}+O(t^i),
\]
and hence
\begin{equation}\label{eq:vanishing-ai}
 a_i(t)\in t^{i-1}R[[t]]
 \qquad(1\le i\le m).
\end{equation}
For a strong coordinate \(X=B(t)\), the transformed law has branches
\(B(z_\nu(B^{-1}(X),B^{-1}(Y)))\). Denote their normalized averaging
operator by \(\widetilde{\cT}_Y=\id+Y\widetilde D+O(Y^2)\), so that
\(\widetilde D\) is its coefficient of \(Y\). Applying \(B\) to each branch gives
\begin{equation}\label{eq:translation-coordinate-change}
 (\widetilde{\cT}_{B(y)}f)(B(t))=\cT_y(f\circ B)(t).
\end{equation}
Taking the coefficient of \(y\), and using \(B'(0)=1\), yields
\begin{equation}\label{eq:generator-coordinate-change}
 (\widetilde Df)(B(t))=D(f\circ B)(t).
\end{equation}
This is the precise change-of-coordinate rule for the generator. Similarly,
\eqref{eq:translation-repetition} shows that repeating all branches \(d\)
times leaves the translation, and therefore its coefficient \(D\), unchanged.
For the cyclic law \(\mathbb G_r\), the coefficient of
\(\dd^r/\dd t^r\) is
\begin{equation}\label{eq:cyclic-top}
 a_r^{\mathrm{cyc}}(t)=\frac{r^r}{r!}t^{r-1};
\end{equation}
see \parencite[Lemma~3.3]{Kholodov1982}.

\begin{lemma}[Formal exponential independence]\label[lemma]{lem:exp-independence}
Let \(L\) be a field of characteristic zero, let \(\alpha_1,\dots,\alpha_s\in L\) be pairwise distinct, and let \(q_1(T),\dots,q_s(T)\in L[T]\) be arbitrary polynomials. If
\begin{equation}\label{eq:exp-relation}
 \sum_{i=1}^s e^{\alpha_iT}q_i(T)=0
 \qquad\text{in }L[[T]],
\end{equation}
then \(q_i=0\) for every \(i\).
\end{lemma}

\begin{proof}
Choose \(d\geq0\) so that every \(q_j\) has degree at most \(d\); such a bound
exists because there are only finitely many polynomials. Fix \(i\). Apply the
constant-coefficient differential operator
\[
 \mathcal{L}_i
 =\prod_{j\neq i}\left(\frac{\dd}{\dd T}-\alpha_j\right)^{d+1}
\]
to \eqref{eq:exp-relation}. For \(j\neq i\), the factor
\[
 \left(\frac{\dd}{\dd T}-\alpha_j\right)^{d+1}
\]
annihilates \(e^{\alpha_jT}q_j(T)\), because it becomes \(e^{\alpha_jT}q_j^{(d+1)}(T)\). Hence
\[
 e^{\alpha_iT}
 \prod_{j\neq i}
 \left(\frac{\dd}{\dd T}+\alpha_i-\alpha_j\right)^{d+1}
 q_i(T)=0.
\]
For a nonzero scalar \(c\in L\), the operator \(\dd/\dd T+c\) is invertible on the finite-dimensional vector space of polynomials of degree at most \(d\): in the monomial basis it is triangular with diagonal entry \(c\). Since \(\alpha_i-\alpha_j\neq 0\), the displayed product is injective on that space. Thus \(q_i=0\).
\end{proof}

\begin{lemma}[A Puiseux solution of the transport equation]\label[lemma]{lem:transport}
Let \((L,\partial_u)\) be a differential field of characteristic zero containing an element \(u\) with \(\partial_u u=1\). Let \(\lambda\in L\) satisfy \(\partial_u\lambda=0\), and let
\[
 w(u,v)\in \bigcup_{N\geq 1}L[[v^{1/N}]],
\]
where \(\partial_u\) acts coefficientwise. If
\begin{equation}\label{eq:transport}
 \partial_vw=\lambda\partial_uw,
 \qquad w(u,0)=u,
\end{equation}
then
\[
 w(u,v)=u+\lambda v.
\]
\end{lemma}

\begin{proof}
Choose \(N\) such that
\[
 w(u,v)=u+\sum_{k\geq 1}c_kv^{k/N},
 \qquad c_k\in L.
\]
If some \(c_k\) with \(1\leq k<N\) were nonzero, choose the least such \(k\). The left-hand side of \eqref{eq:transport} would then have a nonzero term of \(v\)-order \(k/N-1<0\), whereas the right-hand side has no term of negative \(v\)-order. Hence \(c_k=0\) for \(k<N\).

Comparing the coefficient of \(v^0\) in \eqref{eq:transport} gives \(c_N=\lambda\). Subtract \(u+\lambda v\) from \(w\). If the remainder were nonzero, choose its least exponent \(k/N>1\). Differentiation with respect to \(v\) lowers this exponent by one, whereas coefficientwise differentiation with respect to \(u\) leaves it unchanged. The least \(v\)-orders on the two sides would therefore be different, a contradiction. Thus the remainder is zero.
\end{proof}

\begin{lemma}[Ramified substitution and evaluation]\label[lemma]{lem:ramified-evaluation}
Let \(C(t)\in K[[t]]\), let \(Q(t)\in tK[[t]]\) have linear coefficient \(1\), and put
\[
 q(s)=Q(s^m).
\]
Then the substitution
\begin{equation}\label{eq:iota}
 \iota\colon K[[x,y]]\longrightarrow K[[u,v]],
 \qquad x\longmapsto q(u),\quad y\longmapsto q(v),
\end{equation}
is injective.

Let \(L\) be a field containing \(K((u))\), and suppose that
\[
 \widehat z=q(u)+\delta,
 \qquad
 \delta\in v^{1/N}L[[v^{1/N}]]
\]
for some \(N\geq1\). Then \(C(\widehat z)\) is unambiguously defined by the Taylor expansion
\begin{equation}\label{eq:taylor-evaluation}
 C(\widehat z)
 =\sum_{j\geq0}\frac{C^{(j)}(q(u))}{j!}\delta^j
 \in L[[v^{1/N}]].
\end{equation}
The same convention respects sums, products, and composition of one-variable formal series whenever the indicated expressions are defined.
\end{lemma}

\begin{proof}
Write \(U=u^m\) and \(V=v^m\). Substitution by \(Q\) is an isomorphism
\(K[[x,y]]\longrightarrow K[[U,V]]\), defined by
\(x\longmapsto Q(U)\) and \(y\longmapsto Q(V)\), because \(Q\) has a
compositional inverse in \(tK[[t]]\). The homomorphism
\[
 K[[U,V]]\longrightarrow K[[u,v]],
 \qquad U\longmapsto u^m,\quad V\longmapsto v^m,
\]
is injective. Its composite with the preceding isomorphism is
\eqref{eq:iota}.

Since \(q(u)\in u^mK[[u]]\), every coefficient \(C^{(j)}(q(u))\) belongs to \(K[[u]]\subset L\). Because \(\delta\) has strictly positive \(v\)-order, only finitely many summands in \eqref{eq:taylor-evaluation} contribute below any prescribed \(v\)-order. Thus the sum is a well-defined Puiseux power series. The compatibility with the ring operations and composition follows from the formal Taylor identities.
\end{proof}

\section{Proof of the classification}\label{sec:proof-cyclicity}

By \cref{thm:formal-commutativity}, the law is commutative, so all the
Kholodov inputs from \cref{sec:aux-cyclicity} apply.

Retain the notation \(K=\Frac(R)\) and put
\[
 F=\Frac(K[[x,y]]).
\]
Fix a splitting field \(M/F\) for \(P(Z; x, y)\).  In \(M[Z]\)
write
\[
 P(Z;x,y)=\prod_{\nu=1}^{s}(Z-z_\nu(x,y))^{d_\nu},
 \qquad d_\nu\ge1,
 \qquad \sum_{\nu=1}^{s}d_\nu=n,
\]
where \(z_1,\ldots,z_s\) are pairwise distinct.  Characteristic zero makes
\(M/F\) separable, so the derivations \(\partial_x\) and \(\partial_y\)
extend uniquely to \(M\).

\subsection*{Step 1: the translation equation separates branch by branch}\label{subsec:branch-separation}

Let \(D\) be the generator \eqref{eq:generator}, and let \(T\) be a new formal
variable. Apply the translation equation \eqref{eq:translation-equation}
coefficientwise to \(f(Z)=e^{TZ}\).  After multiplying by \(n\), and grouping
equal roots before applying exponential independence, we obtain
\begin{equation}\label{eq:sum-exp}
 \sum_{\nu=1}^s d_\nu
 \left(D_x-D_y\right)e^{Tz_\nu}=0.
\end{equation}
For each branch, repeated use of the chain rule gives
\[
 D_xe^{Tz_\nu}=e^{Tz_\nu}Q_{\nu,x}(T),
 \qquad
 D_ye^{Tz_\nu}=e^{Tz_\nu}Q_{\nu,y}(T),
\]
where \(Q_{\nu,x}\) and \(Q_{\nu,y}\) are polynomials in \(T\) of degree at most \(m\). Their leading coefficients are
\[
 [T^m]Q_{\nu,x}(T)
 =a_m(x)\bigl(\partial_xz_\nu\bigr)^m,
 \qquad
 [T^m]Q_{\nu,y}(T)
 =a_m(y)\bigl(\partial_yz_\nu\bigr)^m.
\]
Indeed, the only contribution to the coefficient of \(T^m\) comes from the top-order term \(a_m\partial^m\) and from differentiating the exponential factor through \(z_\nu\) in every differentiation.

Equation \eqref{eq:sum-exp} is therefore
\[
 \sum_{\nu=1}^s e^{Tz_\nu}d_\nu Q_\nu(T)=0,
 \qquad
 Q_\nu=Q_{\nu,x}-Q_{\nu,y}.
\]
The roots \(z_\nu\) are pairwise distinct, so
\cref{lem:exp-independence} gives \(d_\nu Q_\nu=0\).  Since \(K\) has
characteristic zero, \(d_\nu\ne0\), and hence \(Q_\nu=0\). Comparing the
coefficients of \(T^m\) gives the branchwise identity
\begin{equation}\label{eq:top-branchwise}
 a_m(x)\bigl(\partial_xz_\nu\bigr)^m
 =a_m(y)\bigl(\partial_yz_\nu\bigr)^m
 \qquad (\nu=1,\dots,s).
\end{equation}

\subsection*{Step 2: the top coefficient has exactly cyclic vanishing order}\label{subsec:cyclic-order}

Let
\[
 r=\ord_t a_m(t).
\]
By \eqref{eq:vanishing-ai}, \(r\geq m-1\). We now regard \(x\) as a generic parameter. The natural embedding \(\Frac(K[[x,y]])\longrightarrow K((x))((y))\) extends to an embedding of the chosen splitting field into an algebraic closure. Let \(L\) be an algebraic closure of \(K((x))\). By the Newton--Puiseux theorem, every root of \(P\), viewed as a polynomial over \(L[[y]]\), lies in \(L((y^{1/N}))\) for some \(N\); see \parencite[Chap.~IV, \S2, Prop.~8]{Serre1979}. Because \(P\) is monic, every root has nonnegative \(y\)-order, and reduction of its equation modulo \(y\), using \eqref{eq:formal-unit}, shows that its constant term is \(x\).

No branch is identically equal to \(x\): if \(z_\nu=x\), then \(\partial_xz_\nu=1\) and \(\partial_yz_\nu=0\), contradicting \eqref{eq:top-branchwise}. Hence, for every \(\nu\), there are \(\alpha_\nu\in\mathbb{Q}_{>0}\) and \(b_\nu\in L^\times\) such that
\begin{equation}\label{eq:puiseux-z}
 z_\nu=x+b_\nu y^{\alpha_\nu}+\text{terms of larger \(y\)-order}.
\end{equation}
It follows that
\begin{equation}\label{eq:derivative-orders}
 \ord_y(\partial_xz_\nu)=0,
 \qquad
 \ord_y(\partial_yz_\nu)=\alpha_\nu-1.
\end{equation}
The second equality uses characteristic zero, so that \(\alpha_\nu b_\nu\neq 0\).

Writing \(a_m(t)=ct^r+O(t^{r+1})\) and taking \(y\)-orders in \eqref{eq:top-branchwise}, we obtain
\begin{equation}\label{eq:valuation-equation}
 0=r+m(\alpha_\nu-1).
\end{equation}
Since \(\alpha_\nu>0\), this implies \(r<m\). Together with \(r\geq m-1\), it gives
\begin{equation}\label{eq:r-alpha}
 r=m-1,
 \qquad
 \alpha_\nu=\frac{1}{m}
 \quad\text{for every }\nu.
\end{equation}
Thus
\begin{equation}\label{eq:am-A}
 a_m(t)=ct^{m-1}A(t),
 \qquad c\in K^\times,
 \qquad A(t)\in 1+tK[[t]].
\end{equation}

\subsection*{Step 3: normalize the top coefficient in a ramified coordinate}\label{subsec:ramified-coordinate}

There is a unique series \(H(t)\in 1+tK[[t]]\) satisfying
\begin{equation}\label{eq:H-general}
 H(t)+mtH'(t)=A(t)^{-1/m}.
\end{equation}
Indeed, if \(A(t)^{-1/m}=\sum_{k\geq 0}g_kt^k\) and \(H(t)=\sum_{k\geq 0}h_kt^k\), then
\[
 h_k=\frac{g_k}{1+mk}.
\]
Define
\begin{equation}\label{eq:B-xi}
 B(t)=tH(t)^m,
 \qquad
 \xi(t)=t^{1/m}H(t),
\end{equation}
so that \(\xi(t)^m=B(t)\) and \(B(t)=t+O(t^2)\). Differentiating \(\xi\) and using \eqref{eq:H-general}, we find
\begin{equation}\label{eq:xi-prime}
 \xi'(t)
 =\frac{1}{m}t^{-(m-1)/m}A(t)^{-1/m}.
\end{equation}
Consequently,
\begin{equation}\label{eq:constant-top}
 a_m(t)\bigl(\xi'(t)\bigr)^m=\frac{c}{m^m}.
\end{equation}

Let
\begin{equation}\label{eq:q}
 q(s)=B^{-1}(s^m).
\end{equation}
Then \(q(s)=s^m+O(s^{2m})\) and
\begin{equation}\label{eq:xi-q}
 \xi(q(s))=s.
\end{equation}
Indeed, \(\xi(q(s))^m=B(q(s))=s^m\), while \(\xi(q(s))/s\) has constant term \(1\); the unique \(m\)th root of unity in \(1+sK[[s]]\) is \(1\).
Put
\begin{equation}\label{eq:uv}
 x=q(u),
 \qquad
 y=q(v).
\end{equation}

\subsection*{Step 4: define the normalized roots without ambiguity}\label{subsec:normalized-roots}

Write \(Q=B^{-1}\), so that \(q(s)=Q(s^m)\). The injective substitution
\eqref{eq:iota} extends to an embedding
\[
 \iota\colon F\longrightarrow
 \widehat F:=\Frac(K[[u,v]])\subset K((u))((v)).
\]
Let \(\Omega\) be an algebraic closure of \(K((u))((v))\), and choose an extension
\[
 \widetilde\iota\colon M\longrightarrow\Omega
\]
of this embedding (recall that \(M/F\) is the splitting field for \(P(Z; x, y)\)). The derivations \(\partial_u\) and \(\partial_v\) extend uniquely
from \(K((u))((v))\) to \(\Omega\). Put
\[
 \widehat P(Z;u,v)=P(Z;q(u),q(v)),
 \qquad
 \widehat z_\nu(u,v)=\widetilde\iota\bigl(z_\nu(x,y)\bigr).
\]
Because \(\widetilde\iota\) is a field embedding, the elements
\(\widehat z_1,\ldots,\widehat z_s\) remain pairwise distinct, and
\[
 \widehat P(Z;u,v)
 =\prod_{\nu=1}^{s}(Z-\widehat z_\nu(u,v))^{d_\nu}.
\]

Regard \(u\) as a generic parameter. Since
\[
 \widehat P(Z;u,0)=(Z-q(u))^n,
\]
the Newton--Puiseux theorem \parencite[Chap.~IV, \S2, Prop.~8]{Serre1979} gives, after a finite algebraic extension \(L/K((u))\) and a common ramification \(v^{1/N}\),
\begin{equation}\label{eq:zhat-expansion}
 \widehat z_\nu=q(u)+\delta_\nu,
 \qquad
 \delta_\nu\in v^{1/N}L[[v^{1/N}]].
\end{equation}
Enlarge \(L\), if necessary, so that it contains all \(m\)th roots of unity. The derivation \(\partial_u\) extends uniquely to \(L\), because the extension is separable; every element of \(\muN_m\) is constant for this derivation.
All evaluations of one-variable formal series at \eqref{eq:zhat-expansion} are understood in the precise sense of Lemma~\ref{lem:ramified-evaluation}.

Since \(B(q(u))=u^m\), we can write
\begin{equation}\label{eq:Bzhat-eta}
 B(\widehat z_\nu)=u^m(1+\eta_\nu),
 \qquad
 \eta_\nu\in v^{1/N}L[[v^{1/N}]].
\end{equation}
Define
\begin{equation}\label{eq:w-definition}
 w_\nu(u,v)=u(1+\eta_\nu)^{1/m},
\end{equation}
where the binomial series is chosen with constant term \(1\). Then
\begin{equation}\label{eq:w-properties}
 w_\nu^m=B(\widehat z_\nu),
 \qquad
 w_\nu(u,0)=u.
\end{equation}
Conversely, these two properties characterize \(w_\nu\) uniquely: all other \(m\)th roots specialize at \(v=0\) to \(\zeta u\) with \(\zeta\in\muN_m\) and \(\zeta\neq1\). Finally, because \(q(s)=B^{-1}(s^m)\) and formal evaluation respects composition,
\begin{equation}\label{eq:z-qw}
 \widehat z_\nu=q(w_\nu).
\end{equation}

\subsection*{Step 5: the normalized branches are linear}\label{subsec:linear-branches}

Differentiate \eqref{eq:z-qw}, keeping the other ramified coordinate fixed. The
original and substituted derivations are related by
\[
 \partial_u\widehat z_\nu
 =q'(u)\,\widetilde\iota\bigl(\partial_xz_\nu\bigr),
 \qquad
 \partial_v\widehat z_\nu
 =q'(v)\,\widetilde\iota\bigl(\partial_yz_\nu\bigr).
\]
Indeed, the two derivations
\[
 \partial_u\circ\widetilde\iota
 \qquad\text{and}\qquad
 q'(u)\,\widetilde\iota\circ\partial_x
\]
from \(M\) to \(\Omega\) agree on \(F\), and hence agree on \(M\) because \(M/F\) is
separable. The second identity is obtained in the same way. Since
\(\widehat z_\nu=q(w_\nu)\), we obtain
\begin{equation}\label{eq:chain-x}
 \widetilde\iota\bigl(\partial_xz_\nu\bigr)
 =\frac{q'(w_\nu)}{q'(u)}\,\partial_uw_\nu,
 \qquad
 \widetilde\iota\bigl(\partial_yz_\nu\bigr)
 =\frac{q'(w_\nu)}{q'(v)}\,\partial_vw_\nu.
\end{equation}
Differentiating \eqref{eq:xi-q} gives
\[
 \xi'(q(s))q'(s)=1.
\]
Combining this with \eqref{eq:constant-top}, we obtain
\begin{equation}\label{eq:a-q}
 \frac{a_m(q(s))}{q'(s)^m}=\frac{c}{m^m}.
\end{equation}
Apply \(\widetilde\iota\) to the branchwise identity \eqref{eq:top-branchwise}, substitute \eqref{eq:chain-x}, and then use \eqref{eq:a-q} for \(s=u\) and \(s=v\). Since \(q'(w_\nu)\) has nonzero \(v\)-constant term \(q'(u)\), the common factor
\[
 \frac{c}{m^m}q'(w_\nu)^m
\]
is nonzero and cancels, leaving
\begin{equation}\label{eq:w-powers}
 (\partial_uw_\nu)^m=(\partial_vw_\nu)^m.
\end{equation}
Because \(\partial_uw_\nu(u,0)=1\), the series \(\partial_uw_\nu\) is nonzero and is a unit in the \(v\)-adic Puiseux ring. Therefore
\[
 \rho_\nu:=\frac{\partial_vw_\nu}{\partial_uw_\nu}
\]
is defined and satisfies \(\rho_\nu^m=1\). The coefficient field contains all \(m\)th roots of unity, and the Puiseux field is an integral domain. Since
\[
 T^m-1=\prod_{\lambda\in\muN_m}(T-\lambda)
\]
has distinct roots, we must have
\begin{equation}\label{eq:rho-lambda}
 \rho_\nu=\lambda_\nu
 \quad\text{for some constant }\lambda_\nu\in\muN_m.
\end{equation}
Thus
\[
 \partial_vw_\nu=\lambda_\nu\partial_uw_\nu,
 \qquad
 w_\nu(u,0)=u.
\]
By Lemma~\ref{lem:transport},
\begin{equation}\label{eq:w-linear}
 w_\nu(u,v)=u+\lambda_\nu v.
\end{equation}
Together with \eqref{eq:w-properties}, this yields
\begin{equation}\label{eq:Bz-linear}
 B(\widehat z_\nu(u,v))
 =(u+\lambda_\nu v)^m.
\end{equation}

\subsection*{Step 6: descent in the ramified variables forces equal multiplicities}\label{subsec:count-branches}

If \(\lambda_\nu=\lambda_\mu\), then \eqref{eq:w-linear} gives
\(w_\nu=w_\mu\), and \eqref{eq:z-qw} gives
\(\widehat z_\nu=\widehat z_\mu\).  This contradicts the choice of the
distinct roots when \(\nu\ne\mu\).  Thus the \(\lambda_\nu\) are distinct
elements of \(\muN_m\).

For \(\lambda\in\muN_m\), define its weight by
\[
 e(\lambda)=
 \begin{cases}
 d_\nu,&\lambda=\lambda_\nu\text{ for some }\nu,\\
 0,&\lambda\notin\{\lambda_1,\ldots,\lambda_s\}.
 \end{cases}
\]
Using \eqref{eq:Bz-linear}, the transformed polynomial after the ramified
substitution is
\begin{equation}\label{eq:weighted-normal-form}
 \widehat P_B(W;u,v)
 :=\prod_{\nu=1}^{s}
 \left(W-B(\widehat z_\nu(u,v))\right)^{d_\nu}
 =\prod_{\lambda\in\muN_m}
 \left(W-(u+\lambda v)^m\right)^{e(\lambda)}.
\end{equation}
By the symmetric-series interpretation in \cref{def:strong}, the polynomial
\[
 P_B(W;x,y)=\prod_{\nu=1}^{s}
 \left(W-B(z_\nu(x,y))\right)^{d_\nu}
\]
belongs to \(K[[x,y]][W]\). Equations \eqref{eq:z-qw} and
\eqref{eq:w-linear} give
\(\widehat z_\nu=B^{-1}((u+\lambda_\nu v)^m)\); these are formal series
in \(u,v\) with zero constant term. Thus substitution of \(x=q(u)\) and
\(y=q(v)\) in \(P_B\) gives precisely \(\widehat P_B\).
Since \(q(u)=B^{-1}(u^m)\) and \(q(v)=B^{-1}(v^m)\), its coefficients
belong to \(K[[u^m,v^m]]\). Consequently
\begin{equation}\label{eq:ramified-invariance}
 \widehat P_B(W;\omega u,v)=\widehat P_B(W;u,v)
 \qquad(\omega\in\muN_m).
\end{equation}

The roots \((u+\lambda v)^m\), for \(\lambda\in\muN_m\), are pairwise
distinct: their coefficients of \(u^{m-1}v\) are \(m\lambda\).  Replacing
\(u\) by \(\omega u\) in \eqref{eq:weighted-normal-form} sends the root
indexed by \(\lambda\) to the root indexed by \(\omega^{-1}\lambda\).
Uniqueness of the factorization into these distinct linear factors in \(W\),
together with \eqref{eq:ramified-invariance}, therefore gives
\[
 e(\omega\lambda)=e(\lambda)
 \qquad(\omega,\lambda\in\muN_m).
\]
The multiplication action of \(\muN_m\) on itself is transitive, so every
weight has one common positive value \(d\).  Hence
\begin{equation}\label{eq:n-md}
 s=m,
 \qquad n=md.
\end{equation}

Equation \eqref{eq:weighted-normal-form} now becomes
\begin{equation}\label{eq:identity-after-iota}
 \widehat P_B(W;u,v)=p_m(W;u^m,v^m)^d.
\end{equation}
Since \(u^m=B(q(u))\), \(v^m=B(q(v))\), and the substitution \(\iota\) is
injective by \cref{lem:ramified-evaluation}, this is the image of the identity
\begin{equation}\label{eq:classification-over-K}
 \prod_{\nu=1}^{s}
 \left(W-B(z_\nu(x,y))\right)^{d_\nu}
 =p_m(W;B(x),B(y))^d
\end{equation}
in \(K[[x,y]][W]\).  Thus the classification already holds over \(K\).

\subsection*{Step 7: descent of the strong coordinate to the base ring}\label{subsec:descent-coordinate}

It remains to prove that \(B\), which was constructed over \(K\), has
coefficients in \(R\). Let \(\widetilde D\) be the generator after the
coordinate change \(X=B(t)\). The change-of-coordinate identity
\eqref{eq:generator-coordinate-change} was obtained by taking the coefficient
of \(y\) in the translation identity \eqref{eq:translation-coordinate-change}.
By the chain rule, the coefficient of the highest derivative
\(f^{(m)}(B(t))\) in \(D(f\circ B)(t)\) is \(a_m(t)B'(t)^m\).
Hence
\begin{equation}\label{eq:top-transform}
 \widetilde a_m(B(t))=a_m(t)B'(t)^m.
\end{equation}

By \eqref{eq:classification-over-K}, the transformed law is the \(d\)-fold
diagonal of \(\mathbb G_m\).  Its normalized translation, and hence its
generator, is the same as that of \(\mathbb G_m\).  Therefore
\eqref{eq:cyclic-top} gives
\[
 \widetilde a_m(X)=\frac{m^m}{m!}X^{m-1}.
\]
Since \(B(t)=t+O(t^2)\) and \(B'(t)=1+O(t)\), comparison of the coefficient of
\(t^{m-1}\) in \eqref{eq:top-transform} shows that the constant \(c\) in
\eqref{eq:am-A} is \(m^m/m!\).  The original generator is defined over \(R\),
so \eqref{eq:vanishing-ai} now yields
\[
 A(t)=\frac{m!}{m^m}\frac{a_m(t)}{t^{m-1}}
 \in1+tR[[t]].
\]

Write \(A(t)^{-1/m}=\sum_{k\ge0}g_kt^k\).  The unique solution of
\eqref{eq:H-general} is
\[
 H(t)=\sum_{k\ge0}\frac{g_k}{1+mk}t^k.
\]
All denominators are invertible in the \(\mathbb Q\)-algebra \(R\); hence
\(H\in1+tR[[t]]\), \(B=tH^m\in t+t^2R[[t]]\), and its compositional inverse
also lies in \(R[[t]]\).  Identity \eqref{eq:classification-over-K} is
therefore an identity over \(R\), and it is precisely the strong-isomorphism
identity in \cref{thm:cyclicity}.

Finally, \(B'(t)\) is a unit, so \eqref{eq:top-transform} shows that a
strong change of coordinate preserves the differential order of the generator.
The generator of \(p_r^{\,n/r}\) has rank \(r\); thus \(m\) is unique.  Equation
\eqref{eq:n-md} shows that it is also the number of distinct generic roots.
The roots of \(p_m\) are pairwise distinct generically, so \(p_m^d\) has
zero discriminant exactly when \(d>1\), or equivalently \(m<n\).  This
completes the proof of \cref{thm:cyclicity}. \qed

\section{On weak and strong units}\label{sec:weak-units}

\Cref{thm:algebraic-commutativity} depends on the full multiplicity of the identity
law. A weaker notion is nevertheless natural for algebraic \(n\)-valued
monoids.

\begin{definition}\label[definition]{def:weak-unit}
Let \(X\) be an algebraic variety carrying an associative rational
\(n\)-valued multiplication. An element \(e\in X\) is a \emph{weak unit}
if the multiplication is defined at the generic points of
\(\{e\}\times X\) and \(X\times\{e\}\), and its restrictions satisfy
\[
 x\in e*x
 \qquad\text{and}\qquad
 x\in x*e
\]
for generic \(x\). It is a \emph{strong unit} if
\[
 e*x=x*e=\multiset{x,\ldots,x}.
\]
The variety \(X\) is called an \emph{algebraic \(n\)-valued monoid with a weak
unit} or an \emph{algebraic \(n\)-valued monoid with a strong unit},
respectively.
\end{definition}

Thus every strong unit is weak, while a weak unit records occurrence without
prescribing the remaining values or their multiplicities.

\begin{example}
For simplicity, work over \(\mathbb C\), or over any algebraically closed field of characteristic zero. The Fermat curve gives a basic family. For \(n\geq2\), the polynomial
\[
 P(z;x,y)=z^n-x^n-y^n
\]
defines the commutative associative law (with an operation \(\ast_{\mathbb F_n}\))
\begin{equation}\label{eq:fermat-law}
 x*_{\mathbb F_n}y=
 \multiset{
 \zeta_n^k\sqrt[n]{x^n+y^n}
 \mathbin{|} k=0,\ldots,n-1}
\end{equation}
on \(\mathbb A^1\), where \(\zeta_n = e^{2\pi i/n}\). Both bracketings of three elements are the
roots of \(z^n=x^n+y^n+w^n\), each with the same multiplicity. Moreover,
\[
 0*_{\mathbb F_n}y=y*_{\mathbb F_n}0
 =\multiset{\zeta_n^k y\mathbin{|}k=0,\ldots,n-1}.
\]
Hence \(0\) is a weak unit, but it is not strong when \(n>1\).
\end{example}

There is a general construction behind this example. If \(T\) is a complex
polynomial of degree \(m\), let \(T^{-1}[u]\) denote the \(m\)-multiset of
roots of \(T(z)=u\), counted with multiplicity.

\begin{proposition}[Polynomial pullback]\label[proposition]{prop:polynomial-pullback}
Let \(X=\mathbb A^1_{\mathbb C}\) carry an algebraic \(n\)-valued monoid with
strong unit \(0\), represented by the roots of \(P(z;x,y)\). Let
\(T\in\mathbb C[z]\) have degree \(m\geq2\) and satisfy \(T(0)=0\). Then
\begin{equation}\label{eq:polynomial-pullback}
 a\circ_T b
 =\bigsqcup_{u\in T(a)*T(b)}T^{-1}[u]
 =\multiset{z\mathbin{|}P(T(z);T(a),T(b))=0}
\end{equation}
defines an algebraic \(mn\)-valued monoid with weak unit \(0\). If
the law on \(X\) is commutative, then so is its pullback.
\end{proposition}

\begin{proof}
Each value \(u\in T(a)*T(b)\) has \(m\) preimages under \(T\). On iterating
\eqref{eq:polynomial-pullback}, every value is therefore repeated once for
each preimage of the intermediate point. With multiplicities understood,
\[
 (a\circ_T b)\circ_T c
 =m\cdot\bigsqcup_{w\in (T(a)*T(b))*T(c)}T^{-1}[w]
\]
and
\[
 a\circ_T(b\circ_T c)
 =m\cdot\bigsqcup_{w\in T(a)*(T(b)*T(c))}T^{-1}[w].
\]
The two \(m^2n^2\)-multisets agree by associativity on \(X\). Since \(0\) is
a strong unit on \(X\),
\[
 0\circ_T y=y\circ_T0=n\cdot T^{-1}[T(y)].
\]
This multiset contains \(y\), so \(0\) is a weak unit. Commutativity follows
directly from the symmetry of \eqref{eq:polynomial-pullback} in \(a,b\).
\end{proof}

\begin{example}
The Fermat law \eqref{eq:fermat-law} is the polynomial pullback of ordinary
addition by \(T(z)=z^n\).
\end{example}

The elementary \(n\)-valued group \(\mathbb G_n\) on \(\mathbb A^1\) is
\begin{equation}\label{eq:elementary-law}
 x*_{\mathbb G_n}y
 =\multiset{
 \left(\sqrt[n]{x}+\zeta_n^k\sqrt[n]{y}\right)^n
 \mathbin{|} k=0,\ldots,n-1}.
\end{equation}

Taking \(T(z)=z^m\) in
\cref{prop:polynomial-pullback} gives the
\(mn\)-valued law \(\mathbb G_{m,n}\), defined intrinsically by
\begin{equation}\label{eq:Gmn-polynomial}
 p_n(z^m;x^m,y^m)=0.
\end{equation}
After compatible choices of radicals, its branches may be displayed as
\begin{equation}\label{eq:Gmn-law}
 x*_{\mathbb G_{m,n}}y
 =\multiset{
 \zeta_m^\ell
 \left(\sqrt[n]{x^m}+\zeta_n^k\sqrt[n]{y^m}\right)^{n/m}
 \mathbin{|}
 \substack{k=0,\ldots,n-1\\ \ell=0,\ldots,m-1}},\quad \zeta_m = e^{2\pi i/m}.
\end{equation}
In particular, \(m=n=2\) yields
\[
 x*_{\mathbb G_{2,2}}y=\multiset{\pm x\pm y},
\]
whose four values are the roots of
\[
 z^4+x^4+y^4-2x^2y^2-2x^2z^2-2y^2z^2=0.
\]

The Fermat laws \(\mathbb F_n\) and the pullback laws \(\mathbb G_{n,n-1}\)
are related by the
projective geometry of their defining plane curves. The relevant statement
may be formulated as follows.

\begin{theorem}[Projective duality]\label{thm:weak-unit-duality}
Let \([u:v:w]\) denote the dual coordinates representing the line
\(ux+vy-wz=0\). For \(n\ge2\),
\[
 \{z^n-x^n-y^n=0\}^{\vee}
 =\{p_{n-1}(w^n;u^n,v^n)=0\}.
\]
The curve on the right encodes the \(n(n-1)\)-valued pullback law
\(\mathbb G_{n,n-1}\).
\end{theorem}

See \parencite[Theorem~8]{BuchstaberKornev2025CP1}. With the stated dual
coordinates, the tangent line at \([x:y:z]\) has coordinates
\([x^{n-1}:y^{n-1}:z^{n-1}]\), which gives the displayed polynomial equation.

\printbibliography[heading=bibintoc]
\bigskip

\noindent\begin{minipage}{\linewidth}
\noindent
{\itshape Mikhail Kornev}\\
HSE University\\
Email: \texttt{mkorneff@mi-ras.ru}

\end{minipage}

\end{document}